\documentclass[12 pt]{article}
\usepackage{amsmath, amsthm, amssymb,algorithm2e,caption,enumerate}
\usepackage{tikz}
\usetikzlibrary{shapes,snakes}
\usepackage{fullpage}
\usepackage{enumitem}
\usepackage{hyperref}
\usepackage{xcolor}

\title{Stability results for graphs with a critical edge}
\author{Alexander Roberts\thanks{Mathematical Institute, University of Oxford, Andrew Wiles Building, Radcliffe Observatory Quarter, Woodstock Road, Oxford, United Kingdom. \newline  E-mail: \texttt{\{robertsa, scott\}@maths.ox.ac.uk}.} \and Alex Scott\footnotemark[1]}

\newtheoremstyle{case}{}{}{\normalfont}{}{\itshape}{:}{ }{}

\newtheorem{thm}{Theorem}[section]
\newtheorem{lem}[thm]{Lemma}
\newtheorem{prop}[thm]{Proposition}

\theoremstyle{definition}
\newtheorem{defn}[thm]{Definition}

\numberwithin{equation}{section}

\newtheoremstyle{case}{}{}{\normalfont}{}{\itshape}{\normalfont:}{ }{}

\theoremstyle{case}

\def\comment#1{}

\newcommand{\ex}{{\rm ex}}
\newcommand{\Ex}{{\rm Ex}}
\newcommand{\beq}{\begin{eqnarray*}}
\newcommand{\eeq}{\end{eqnarray*}}

\def\build#1_#2^#3{\mathrel{\mathop{\kern 0pt#1}\limits_{#2}^{#3}}}
\newcommand{\beqs}{\begin{eqnarray}}
\newcommand{\eeqs}{\end{eqnarray}}

\newcommand{\bN}{\mathbb{N}}

\numberwithin{equation}{section}

\tikzset{
    position/.style args={#1:#2 from #3}{
        at=(#3.#1), anchor=#1+180, shift=(#1:#2)
    }
}
\begin{document}

\maketitle

\begin{abstract}
The classical stability theorem of Erd\H{o}s and Simonovits states that, for any fixed graph with chromatic number $k+1 \ge 3$, the following holds: every $n$-vertex graph that is $H$-free and has within $o(n^2)$ of the maximal possible number of edges can be made into the $k$-partite Tur\'{a}n graph by adding and deleting $o(n^2)$ edges.   
In this paper, we prove sharper quantitative results for graphs $H$ with a critical edge, both for the Erd\H{o}s-Simonovits Theorem (distance to the Tur\'{a}n graph) 
and for the closely related question of how close an $H$-free graph is to being $k$-partite.
In many cases, these results are optimal to within a constant factor.
\end{abstract}

\section{Introduction}
\tikzstyle{vertex} = [fill,shape=circle,node distance=80pt]
\tikzstyle{vertex2} = [draw,shape=circle,node distance=80pt,text width={width("$M^1_1$")},align=center]
\tikzstyle{hex} = [draw,regular polygon, regular polygon sides = 6]
\tikzstyle{switch} = [fill,shape=circle,node distance=80pt]
\tikzstyle{vsquare} = [draw,shape=rectangle,node distance = 80pt]
\tikzstyle{edge} = [fill,opacity=.3,fill opacity=.3,line cap=round, line join=round, line width=30pt]

For $n,k\ge1$, the {\em $k$-partite Tur\'{a}n graph $T_k(n)$} is the complete $k$-partite graph on $n$ vertices with vertex classes as equal as possible (or equivalently the $k$-partite graph with maximum number of edges).
We write $t_k(n)=e(T_k(n))$ for the number of edges in the Tur\'{a}n graph.
A fundamental result in extremal graph theory is the Erd\H{o}s-Simonovits Stability Theorem, which says that an $H$-free graph that is close to extremal must in fact look very much like a Tur\'{a}n graph.

\begin{thm}[Erd\H{o}s-Simonovits \cite{erd-sim}]\label{e-s}
Let $k \ge 2$ and suppose that $H$ is a graph with $\chi(H) = k+1$. If $G$ is an $H$-free graph with $e(G) \ge t_k(n) - o(n^2)$, then $G$ can be formed from $T_k(n)$ by adding and deleting $o(n^2)$ edges.
\end{thm}

It is natural to ask how the $o(n^2)$ terms here depend on each other.  Thus we will consider an $H$-free graph $G$ with $n$ vertices and $t_k(n)-f(n)$ edges, where $f(n)=o(n^2)$, and ask how close $G$ is to the Tur\'{a}n graph $T_k(n)$.

In this paper, we will be interested in the case when $H$ has a critical edge:
an edge $e$ in a graph $H$ is said to be \emph{critical} if $\chi(H - e) = \chi(H) - 1$.  It was shown by Simonovits \cite[Theorem 2.3]{Sim1} that if $H$ has a critical edge, then for sufficiently large $n$ the Tur\'an graph $T_k(n)$ is the unique extremal $H$-free graph on $n$-vertices (while if $H$ does not have a critical edge, then the Tur\'{a}n graph is not extremal). In this case, we will prove the following version of the Erd\H{o}s-Simonovits Theorem.

\begin{thm}\label{escrit}
Let $H$ be a graph with a critical edge and $\chi(H)$ = $k+1 \ge 3$, and let $f(n) = o(n^2)$ be a function.
If $G$ is an $H$-free graph with $n$ vertices and  $e(G) \ge t_k(n) - f(n)$, then $G$ can be formed from $T_k(n)$ by adding and deleting $O(f(n)^{1/2}n)$ edges.
\end{thm}

There is a simple construction showing that this bound is sharp up to a constant factor: if we take the Tur\'{a}n graph $T_k(n)$ and imbalance it by moving $\lceil f(n)^{1/2} \rceil$ vertices from one class to another then we obtain a $k$-partite graph $G$ with $n$ vertices and $t_k(n)-\Theta(f(n))$ edges.  However, in order to obtain $T_k(n)$ from $G$ we must change at least $\Omega(f(n)^{1/2}n)$ edges.

Theorem \ref{escrit} will follow from a result on the closely related question: how many edges do we need to delete from $G$ in order to make it $k$-partite? We will say that a graph $G$ is \em $r$ edges away from being $k$-partite \em if the largest $k$-partite subgraph of $G$ has $e(G) -r$ edges. In a recent paper, F{\"u}redi \cite{Furedi} gave a beautiful proof of the following result.\footnote{F{\"u}redi also claims that Gy\H{o}ri's work \cite{gyori} implies the bound $O(f(n)^2n^{-2})$. But F{\"u}redi's claim is not correct, and Proposition \ref{counter1} below shows that the bound $O(f(n)^2n^{-2})$ is not in general valid.}

\begin{thm}[F{\"u}redi \cite{Furedi}]\label{furedi}
Suppose that $G$ is a $K_{k+1}$-free graph on $n$ vertices with $t_k(n)-t$ edges.  Then $G$ can be made $k$-partite by deleting at most $t$ edges.
\end{thm}

We will show that a much stronger bound holds. More generally, we will prove results for graphs $H$ that contain a critical edge (note that every edge of $K_{k+1}$ is critical). As we will see below, our bounds are sharp to within a constant factor for many graphs $H$ (including $K_{k+1}$).

\begin{thm}\label{stable}
Let $H$ be a graph with a critical edge and $\chi(H) = k+1 \ge 3$, and let $f(n) = o(n^2)$ be a function.
If $G$ is an $H$-free graph with $n$ vertices and $e(G) \ge t_k(n) - f(n)$ then $G$ can be made $k$-partite by deleting $O(n^{-1}f(n)^{3/2})$ edges.
\end{thm}

As we will see below, for many $H$ (including $K_{k+1}$) the bound in Theorem \ref{stable} is optimal up to a constant factor; in many other cases we will be able to prove a stronger bound. In order to discuss this, we will need some definitions.

For disjoint sets $A$, $B$ of vertices we write $K[A,B]$ for the edge set $\{ab:a \in A, b\in B\}$ of the complete bipartite graph $K_{A,B}$.  
For a graph $G = (V,E)$, recall that the {\em Mycielskian} \cite{Myc} of $G$ is a graph $M(G)$ with vertex set	
$ V \cup V' \cup \{u\}$ (where $V' = \{v' : v \in V\}$) and edge set
$ E \cup \{vw' : vw \in E\} \cup K[V',\{u\}]$. Informally, the Mycielskian of a graph $G$ is the graph attained by adding a copy $v'$ of each vertex $v$ in $G$ (where $v'$ is adjacent to $\Gamma_G(v)$, but not to copies of other vertices) and then adding a new vertex adjacent to all copies. For example, the Mycielskian of an edge is the pentagon.

We define the
{\em blown-up Mycielskian graph $M_k(a,b,c)$} as follows.  Let $V_1,\dots,V_k$ be sets of size $a$, let 
$W_1,\dots,W_k$ be sets of size $b$, and let $U$ be a set of size $c$ (and let all these sets be disjoint).  Then $M_k(a,b,c)$ has vertex set
$\bigcup_{i=1}^kV_i\cup\bigcup_{i=1}^kW_i\cup U$ and edge set 
$$\bigcup_{i \neq j} K[V_i, V_j] \cup \bigcup_{i \neq j} K[V_i, W_j] \cup \bigcup_{i} K[W_i,U].$$
Note that $M_k(a,b,c)$ is a blowup of the graph $M(K_k)$. Indeed, from $M(K_k)$ (with vertex set $V\cup V'\cup\{u\}$), one obtains $M_k(a,b,c)$ by taking $a$ copies of each vertex in $V$, then $b$ copies of each vertex in $V'$, and then $c$ copies of $u$.

\begin{minipage}{0.45\textwidth}
\centering
\begin{tikzpicture}[thick,scale=0.6, every node/.style={scale=0.6}]
	\node[label = $u$,vertex] at (0,0) (x2) {} ;
	\node[label = $v_1$,vertex,position=180:{3.5} from x2] (y4) {} ;
	\node[label = $v_2$,vertex,position=90:{2.2} from y4] (y2) {} ;
	\node[label = below:{$v_3$},vertex,position=90:{-2.8} from y4] (y6) {} ;
	\node[position=180:{3.5} from y4] (dr) {} ;
	\node[label = $v'_1$,vertex,position=180:{1} from dr] (z2) {} ;
	\node[label = $v'_2$,vertex,position=90:{1.7} from dr] (z1) {} ;
	\node[label = below:{$v'_3$},vertex,position=90:{-2.2} from dr] (z3) {} ;
	\draw (y2) -- (z2) ;
	\draw (y2) -- (z3) ;
	\draw (y4) -- (z1) ;
	\draw (y4) -- (z3) ;
	\draw (y6) -- (z2) ;
	\draw (y6) -- (z1) ;
	\draw (z1) -- (z2) ;
	\draw (z2) -- (z3) ;
	\draw (z3) -- (z1) ;
	\draw (x2) -- (y2) ;
	\draw (x2) -- (y4) ;
	\draw (x2) -- (y6) ;
\end{tikzpicture}
\captionsetup{font=footnotesize}

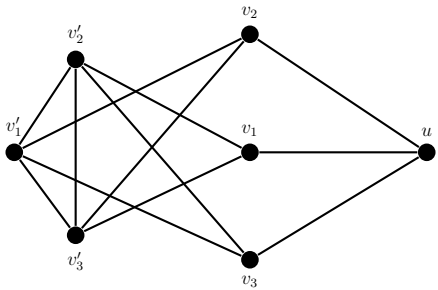
\captionof{figure}{Mycielskian of $K_3$ ($M_3(1,1,1)$)}\label{neartriangle}
\end{minipage}
\begin{minipage}{0.45\textwidth}
\centering
\begin{tikzpicture}[thick,scale=0.6, every node/.style={scale=0.6}]
	\node at (0,0) (x2) {} ;
	\node[vertex,position=90:{0.8} from x2] (x1) {} ;
	\node[vertex,position=90:{-1.2} from x2] (x3) {} ;
	\node[position=180:{4} from x2] (dl) {} ;
	\node[vertex,position=90:{0.2} from dl] (y3) {} ;
	\node[vertex,position=90:{2.2} from dl] (y2) {} ;
	\node[vertex,position=90:{3.2} from dl] (y1) {} ;
	\node[vertex,position=90:{-0.8} from dl] (y4) {} ;
	\node[vertex,position=90:{-2.8} from dl] (y5) {} ;
	\node[vertex,position=90:{-3.8} from dl] (y6) {} ;
	\node[position=180:{4} from dl] (dr) {} ;
	\node[vertex,position=180:{1} from dr] (z2) {} ;
	\node[vertex,position=90:{1.7} from dr] (z1) {} ;
	\node[vertex,position=90:{-2.2} from dr] (z3) {} ;
	\draw (y1) -- (z2) ;
	\draw (y1) -- (z3) ;
	\draw (y2) -- (z2) ;
	\draw (y2) -- (z3) ;
	\draw (y3) -- (z1) ;
	\draw (y3) -- (z3) ;
	\draw (y4) -- (z1) ;
	\draw (y4) -- (z3) ;
	\draw (y5) -- (z2) ;
	\draw (y5) -- (z1) ;
	\draw (y6) -- (z2) ;
	\draw (y6) -- (z1) ;
	\draw (z1) -- (z2) ;
	\draw (z2) -- (z3) ;
	\draw (z3) -- (z1) ;
	\draw (x1) -- (y1) ;
	\draw (x1) -- (y2) ;
	\draw (x1) -- (y3) ;
	\draw (x1) -- (y4) ;
	\draw (x1) -- (y5) ;
	\draw (x1) -- (y6) ;
	\draw (x3) -- (y1) ;
	\draw (x3) -- (y2) ;
	\draw (x3) -- (y3) ;
	\draw (x3) -- (y4) ;
	\draw (x3) -- (y5) ;
	\draw (x3) -- (y6) ;
\end{tikzpicture}
\captionsetup{font=footnotesize}

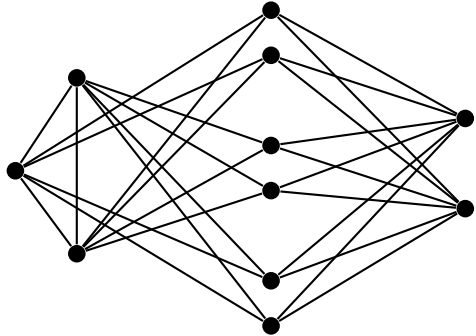
\captionof{figure}{$M_3(1,2,2)$}\label{neartriangle3}
\end{minipage}

The optimal error bound in Theorem \ref{stable} turn out to depend on whether $H$ is a subgraph of some blown-up Mycielskian graph $M_k(a,b,c)$ (note that here it is enough just to consider the case $a = b = c = |H|$. If $H$ is not contained in one of these blow-ups, then the bound in Theorem \ref{stable} is tight up to a constant factor.

\begin{prop}\label{counter1}
Let $H$ be a graph with a critical edge and $\chi(H)$ = $k+1 \ge 3$, and let $f(n) = o(n^2)$ be a function with $f(n) \ge 2n$. Suppose that $H$ is not a subgraph of $M_k(a,a,a)$ where $a=|V(H)|$. Then there is an $H$-free graph $G$ with $n$ vertices and at least $t_k(n) - f(n)$ edges which cannot be made $k$-partite by deleting $o(n^{-1}f(n)^{3/2})$ edges.
\end{prop}

It is natural to ask what happens when $f$ is very small. It follows from a result of Simonovits \cite[p. 282]{Sim4} that for sufficiently large $n$ and $f(n) < \frac{n}{k} - O(1)$, any $H$-free graph with $n$ vertices and at least $t_k(n) - f(n)$ edges is already $k$-partite. On the other hand, the bound in Theorem \ref{escrit} is still sharp in this range. Indeed, as noted above, unbalancing the class sizes in the Tur\'{a}n graph so that one class is $f(n)^{1/2}$ larger than the rest gives a graph with $e(G) \ge t_k(n) - f(n)$ which requires the addition and deletion of $\Theta(f(n)^{1/2}n)$ edges to form $T_k(n)$.

If $H$ {\em is} a subgraph of some blown-up Mycieslkian $M_k(a,b,c)$ then the construction used to prove Proposition \ref{counter1} can no longer be used.  However, we do have the following general lower bound
that holds for all graphs $H$ with a critical edge.

\begin{prop}\label{propcount1}
Let $H$ be a graph with a critical edge and $\chi(H)$ = $k+1 \ge 3$, and let $f(n) = o(n^2)$ be a function with $f(n) \ge 2n$. Then there is an $H$-free graph $G$ with $n$ vertices and at least $t_k(n) - f(n)$ edges which cannot be made $k$-partite by deleting $o(n^{-2}f(n)^2)$ edges.
\end{prop}

Note that this is much weaker than the bound given in Proposition \ref{counter1}.  However, if $H$ is contained in $M_k(a,a,1)$ for some $a$, then the bound in Theorem \ref{stable} can be substantially strengthened. Indeed, for this class of graphs it turns out that Proposition \ref{propcount1} is in fact tight to within a constant factor.

\begin{thm}\label{stable2}
Let $H$ be a graph with a critical edge and $\chi(H)$ = $k+1 \ge 3$, and 
suppose that $H$ is a subgraph of $M_k(a,a,1)$ for some $a$.
Let $f(n) = o(n^2)$ be a function.
If $G$ is an $H$-free graph on $n$ vertices with $e(G) \ge t_k(n) - f(n)$ then $G$ can be made $k$-partite by deleting $O(n^{-2}f(n)^2)$ edges.
\end{thm}

Theorems \ref{stable} and \ref{stable2} give bounds that are sharp to within a constant factor when $H$ is a subgraph of some $M_k(a,a,1)$ or when $H$ is not contained in any $M_k(a,a,a)$.  
What about graphs that are contained in some $M_k(a,a,a)$ but are not contained in any $M_k(a,a,1)$?  In this case, we do not have sharp results, but can say a little.

\begin{thm}\label{stableII}
Let $H$ be a graph with a critical edge and $\chi(H)$ = $k+1 \ge 3$, and 
suppose that $H$ is a subgraph of $M_k(t,b,a)$.
Let $f(n) = o(n^2)$ be a function.
If $G$ is an $H$-free graph on $n$ vertices such that $e(G) \ge t_k(n) - f(n)$ then $G$ can be made $k$-partite by deleting $O(n^{\frac{1}{bk}}f(n)^{1-\frac{1}{bk}})$ edges.
\end{thm}

Note that the bound in Theorem \ref{stableII} is stronger than the bound in Theorem \ref{stable} when $f(n) \gg n^{2-\frac{2}{bk+2}}$. This shows that the upper bound $O(n^{-1}f(n)^{3/2})$ in Theorem \ref{stable} isn't tight when $H$ is contained in some $M_k(a,a,a)$ but is not contained in any $M_k(a,a,1)$. However, we will give examples in Section \ref{gap} showing that these graphs need not satisfy the stronger $O(n^{-2}f(n)^2)$ bound of Theorem \ref{stable2}. We discuss this further in the conclusion.

The paper is organised as follows. In Section \ref{sec3} we give proofs of Theorems \ref{escrit}, \ref{stable}, \ref{stable2} and \ref{stableII}. In Section \ref{countergraphs}, we will prove Propositions \ref{counter1} and \ref{propcount1} by way of constructions. In Section \ref{gap} we discuss the gap between the upper bound given by Theorem \ref{stableII} and the lower bound given by Proposition \ref{propcount1}, and conclude the paper with some related problems and open questions. 

Related results can be found in papers by Norin and Yepremyan \cite{Liana, Liana2}, and Pikhurko, Sliacan and Tyros \cite{Oleg}. Finally, we note that results from this paper are applied in a joint paper with Natasha Morrison \cite{cycles}.

For the duration of this paper, we write $a(n) = o(b(n))$ to mean $\frac{a(n)}{b(n)} \rightarrow 0$ as $n \rightarrow \infty$. We also use the notation $a(n) = O(b(n))$ if there is some constant $C$ such that $|a(n)| \le C|b(n)|$ for all $n$, and $a(n) = \Omega(b(n))$ if $b(n) = O(a(n))$.

\section{Upper Bounds}\label{sec3}
In this section we present our proofs of Theorems \ref{escrit}, \ref{stable}, \ref{stable2} and \ref{stableII}. We start by recalling some important results that will be key to our argument. The first is the Erd\H{o}s-Stone Theorem \cite{erd-stone} concerning the extremal number of a complete symmetric $k$-partite graph.

\begin{thm}[Erd\H{o}s-Stone \cite{erd-stone}]
\label{erdstone}
Let $k \ge 2$, $t \ge 1$, and $\epsilon > 0$. Then for $n$ sufficiently large, if $G$ is a graph on $n$ vertices with
$$ e(G) \ge \left(1 - \frac{1}{k-1} + \epsilon\right)\binom{n}{2},$$ 
then $G$ must contain a copy of $T_k(kt)$.
\end{thm}

The second is a theorem proven by Simonovits regarding the extremal graph of a graph with a critical edge.

\begin{thm}[Simonovits {\cite[Theorem 2.3]{Sim1}}]
\label{simthm}
Let $H$ be a graph with a critical edge with $\chi(H) =k+1 \ge 3.$ Then there exists some $n_0$ such that, for all $n \ge n_0$, we have $\Ex(n;H) = \{T_k(n)\}$.
\end{thm}

As a stepping stone to Theorem \ref{stable}, we first prove a weaker result. This can also be deduced from independent work of Norin and Yepremyan (\cite[Theorem 3.1]{Liana} and \cite{Liana2}) but we include a self-contained proof here for completeness. The approach is essentially standard (see for example Erd\H{o}s \cite{erd-sim}), but we must keep careful track of the error bounds.

\begin{lem}\label{stableb}
Let $H$ be a graph with a critical edge and $\chi(H)$ = $k+1\ge 3$, and let $f(n) = o(n^2)$ be a function. Suppose that $G$ is an $H$-free graph on $n$ vertices such that $e(G) \ge t_k(n) - f(n)$. Then $G$ can be made $k$-partite by deleting $O(f(n))$ edges.
\end{lem}

The following easy lemma will be used repeatedly in the proof of Lemma \ref{stableb}.

\begin{prop}\label{crux}
Fix $k \ge 2$ and $t \ge 1$. For $n \ge 1$, suppose $G \subset T_k(kn)$ is a $T_k(kt)$-free graph. Then for sufficiently large $n$,
	\beqs
		e(G) \le t_k(kn) - \frac{n^2}{2}.
	\eeqs
\end{prop}

\begin{proof}
Let $k \ge 2$ and $n,t \in \bN$ and suppose $G \subset T_k(kn)$ is $T_k(kt)$-free. Then $e(G) \le \ex(kn;T_k(kt))$, so we can apply Theorem \ref{erdstone} to get
	\beqs
		t_k(kn) - e(G) &\ge& t_k(kn) - \ex(kn;T_k(kt)) \nonumber \\
		&\ge& \frac{(kn)^2}{2}\left(1 - \frac{1}{k}\right) - \frac{(kn)^2}{2}\left(1 - \frac{1}{k-1} + o(1)\right) \nonumber \\
		&\ge& \frac{n^2}{2}, \nonumber
	\eeqs
for sufficiently large $n$.
\end{proof}

\begin{proof}[\bf Proof of Lemma \ref{stableb}]
Let $f(n) = o(n^2)$ be a function, and let $H$ be a graph with a critical edge and $\chi(H)$ = $k+1 \ge 3$. Choose $t$ such that $H \subseteq T_k(tk) + e$, where $e$ is any edge inside a vertex class of $T_k(tk)$. Let $\delta, \epsilon, \eta \in (0,\frac{1}{20000k^2})$ with $\epsilon \le \frac{\eta^2}{8}$.

Suppose that $G$ is an $H$-free graph on $n$ vertices with $e(G) \ge t_k(n) - f(n)$. Since $f(n) = o(n^2)$, Theorem \ref{e-s} tells us that there exists some $N_0$ such that when $n \ge N_0$, $G$ is at most $\epsilon n^2$ edges away from a complete $k$-partite graph. By Theorem \ref{simthm} we may assume that $\Ex(n ; H) = \{T_k(n)\}$ for all $n \ge N_0$.

Now suppose $n \ge 2N_0$ and let $L\subsetneq V(G)$ be the set of vertices with degree less than $(1-\delta)\frac{n(k-1)}{k}$. Consider an arbitrary subset $B \subset L$ with $|B| < \frac{\delta n}{2}$ and let $J = G \setminus B$. We can count the number of edges in $J$ by considering the number of edges removed from $G$ to get that
	\beqs
		e(J) &\ge& t_k(n) - f(n) - (1-\delta)n\left(\frac{k-1}{k}\right)|B|. \label{Jcount1}
	\eeqs
Now $t_k(n) \ge t_k(n-1) + \frac{k-1}{k}(n-1)$, since we can form $T_k(n)$ from $T_k(n-1)$ by adding a vertex to a smallest vertex class, and so $t_k(n) \ge t_k(n-|B|) + \frac{k-1}{k}(n-|B|)|B|$. If we apply this inequality to \eqref{Jcount1}, we see that
	\beqs
		e(J) &\ge& t_k(|J|) - f(n) + \left(\frac{k-1}{k}\right)\left(|B|(n-|B|) - (1-\delta)n|B|\right) \nonumber \\
		&=& t_k(|J|) - f(n) + \left(\frac{k-1}{k}\right)\left(\delta n|B| - |B|^2 \right) \nonumber \\
		&\ge& t_k(|J|) - f(n) + \left(\frac{k-1}{2k}\right)\delta n|B|. \label{erd1}
	\eeqs
On the other hand, $J$ does not contain a copy of $H$ and $|J| > n-\frac{\delta}{2} n > \frac{n}{2} \ge N_0$, so $e(J) \le t_k(|J|)$. Comparing this upper bound with the lower bound given by \eqref{erd1}, we see that
	\beqs
		|B| \le \frac{2k}{\delta(k-1)}f(n)n^{-1}. \label{Jcount2}
	\eeqs
Recall that $f(n) = o(n^2)$ and so $\frac{2k}{\delta(k-1)}f(n)n^{-1} < \frac{\delta n}{4}$ for large enough $n$. Since $B$ is an arbitrary subset of $L$ with $|B| < \frac{\delta n}{2}$, we can conclude that $|L| < \frac{2k}{(k-1)\delta}f(n)n^{-1}$ for sufficiently large $n$.

Now fix $J=G \setminus L$. Since $|L| <  \frac{2k}{(k-1)\delta}f(n)n^{-1}$, we have lost at most $\frac{2}{\delta}f(n) = O(f(n))$ edges, so it suffices to show that the graph $J$ must be $k$-partite. 

Let $q := |J| = (1+ o(1))n$. For sufficiently large $n$, the graph $J$ has minimum degree at least $(1-2\delta)\frac{q(k-1)}{k}$ and has $e(J) = t_k(q) - O(f(n))$. Furthermore, we already know that $J$ is at most $\epsilon n^2 = \epsilon q^2 (1+o(1))$ edges away from being $k$-partite since $J$ is a subgraph of $G$. We may then choose a partition $V_1,\ldots, V_k$ of $V(J)$  which contains at most $\epsilon n^2$ edges within the vertex classes.  For $n$ sufficiently large, since $f(n) = o(n^2)$, there are at most $\frac{3}{2}\epsilon n^2$ edges missing between the vertex classes. We now use this fact to derive information about the vertices in $J$.

Suppose some vertex $v$ has at least $\eta n$ neighbours in each vertex class. Pick $\eta n$ neighbours of $v$ in each vertex class to form $ Q \subseteq V(J)$. Now let $P$ be the subgraph of $J[Q]$ obtained by deleting all edges inside the classes $Q \cap V_i$. Note that if $P$ contains a copy of $T_k(kt)$, then $J[Q \cup \{v\}]$ contains a copy of $H$, contradicting the fact that $J$ is $H$-free. Therefore $P$ is $T_k(tk)$-free. 
An application of Proposition \ref{crux} then gives, for $n$ sufficiently large,
	\beqs
		e(P) &\le& t_k(kn) - \frac{\eta ^2}{2}n^2. \label{example2}
	\eeqs
Since $\epsilon \le \frac{\eta^2}{8}$, we get $e(P) \le t_k(k\eta n) - 4\epsilon n^2$. But then at least $4\epsilon n^2$ edges between vertex classes are not present in $J$, which gives a contradiction. We may therefore assume that every vertex in $J$ has at most $\eta n$ neighbours inside its own vertex class.

If $|V_i| - |V_j| \ge \frac{q}{50k}$ for some $j$, then $e(J) \le t_k(q) - \frac{q^2}{10000k^2} + \epsilon q^2(1+o(1)) < t_k(q) - \frac{\epsilon}{2}n^2$ for sufficiently large $n$, since $\epsilon < \frac{1}{20000k^2}$. This is impossible as $f(n) = o(n^2)$ and $e(J) = t_k(q) - O(f(n))$. So we may asume that $|V_i| \ge \frac{q}{k}(1-\frac{1}{50})$ for each $i$.

Suppose, without loss of generality, that there is an edge $uv$ inside $V_1$.  Consider the neighbourhoods of $u$ and $v$ in each of the vertex classes $V_2,\ldots V_k$. Note that $|\Gamma(u)\cap(V_2\cup \cdots \cup V_k)| \ge q(1-2\delta)(\frac{k-1}{k}) - 2\eta n$. At the same time $|V_2 \cup \cdots \cup V_k| \le q(1-\frac{1}{k}(1-\frac{1}{50}))$ and so $|(V_2 \cup \cdots \cup V_k) \setminus \Gamma(u)| \le q(\frac{1}{50k} + 2\delta + 2 \eta) \le \frac{q}{10k}$. The same argument applies for $v$ and so there are most $\frac{q}{5k}$ vertices not in $\Gamma(u)\cap \Gamma(v)$ in each vertex class $V_2,\ldots,V_k$. So $|\Gamma(u) \cap \Gamma(v) \cap V_i| \ge \frac{q}{k}(1-\frac{1}{50}) - \frac{q}{5k} \ge \frac{q}{2k} \ge \eta n$ for each $i \ge 2$.

Pick $S_1 \subset V_1$ and $S_i \subset \Gamma(u) \cap \Gamma(v) \cap V_i$ for each $i =2,\ldots,k$ with $|S_i| = \eta n$ for each $i$. Let $Q = S_1 \cup \ldots \cup S_k$ and $P$ be the subgraph of $J[Q]$ obtained by deleting all edges inside the $S_i$. Arguing as in \eqref{example2}, we get $e(P) < t_k(k\eta n) - 4\epsilon n^2$ and so arrive at the same contradiction. Therefore there is no edge $uv$ inside $V_1$ and so $J$ must be $k$-partite as required.
\end{proof}

We are now in a position to prove Theorem \ref{stable}. We have to work rather harder, and our argument is guided by the structure of the examples showing lower bounds.

\begin{proof}[\bf Proof of Theorem \ref{stable}]
Let $f(n) = o(n^2)$ be a function, and let $H$ be a graph with a critical edge and $\chi(H)$ = $k+1$. Choose $t$ such that $H \subseteq T_k(tk) + e$, where $e$ is any edge inside a vertex class of $T_k(tk)$. Let $G$ be an $H$-free graph on $n$ vertices with $e(G) \ge t_k(n) - f(n)$. Take a partition $(V_1,\ldots,V_k)$ of $V(G)$ which minimises the number of edges inside vertex classes. Then by Lemma \ref{stableb}, there are $O(f(n))$ edges within vertex classes and at most $O(f(n))$ edges between vertex classes are not present in $G$. Furthermore for each $i$, we must have $|V_i| = (1+o(1))\frac{n}{k}$, otherwise $G$ cannot contain enough edges.

Let $v$ be a vertex in $G$ with maximal number of neighbours inside its own vertex class. Without loss of generality, we may assume $v \in V_1$. Let $r(v) = |\Gamma(v) \cap V_1|$. By our choice of partition, $v$ has at least $r(v)$ neighbours in every other vertex class. Pick $r(v)$ neighbours of $v$ from each vertex class to form $Q \subseteq V(G)$ and let $J$ be the subgraph of $G[Q]$ with all the edges within vertex classes removed (so $J$ is $k$-partite). If $J$ contains a copy of $T_k(tk)$, then by adding $v$ to this copy, we must get a copy of $H$. $J$ must then be $T_k(tk)$-free and so we may apply Proposition \ref{crux} to get that if $r(v)$ is sufficiently large, then
	\beqs
		e(J) &\le& t_k(kr(v)) -  \frac{r(v)^2}{2}. \label{example1}
	\eeqs
As $G$ is missing at most $O(f(n))$ edges between vertex classes, we must have $t_k(kr(v))-e(J) = O(f(n))$ and so $r(v) = O(f(n)^{1/2})$.

Let $\delta \in (0,\frac{1}{4tk^2})$ and $S$ be the set of vertices in $V(G)$ with degree less than $n(1-\frac{1}{k})(1-\delta)$. Arguing as in Lemma \ref{stableb} (around \eqref{erd1}) we see that $|S| = O(f(n)n^{-1})$. We will show that each edge inside a vertex class is incident to a vertex in $S$.

Let $E_I$ be the set of edges inside vertex classes. Pick some edge $e=uv \in E_I$ and suppose that neither $u$ nor $v$ is an element of $S$. Without loss of generality, assume that $u,v \in V_1$. Recall that $|V_i| = (1+o(1))\frac{n}{k}$ for each $i$. Since $|\Gamma(u)| \ge n(1-\frac{1}{k})(1-\delta) \ge \frac{k-1}{k}n - \frac{n}{20k}$ and $|V_2 \cup \ldots \cup V_k| = (1+o(1))\frac{n(k-1)}{k}$, it follows that $|V_i \setminus \Gamma(u)| \le \frac{n}{20k}(1+o(1))$ for each $i \in \{2,\ldots,k\}$. The same is true for $v$. So for sufficiently large $n$, $|V_i \cap \Gamma(u) \cap \Gamma(v)| \ge \frac{n}{2k}$ for each $i \in \{2,\ldots,k\}$. Pick $B_1 \subset V_1$ and $B_i \subset \Gamma(u) \cap \Gamma(v) \cap V_i$ for each $i =2,\ldots,k$ with $|B_i| = \frac{n}{2k}$ for each $i$. If $Q = G[B_1 \cup \ldots \cup B_k]$ contains a copy of $T_k(kt)$, then $G[Q \cup \{u,v\}]$ contains a copy of $H$, contradicting the fact that $J$ is $H$-free. Therefore $Q$ is $T_k(tk)$-free. So by Proposition \ref{crux},
	\beqs
		e(Q) &\le& t_k(\frac{n}{2k}) - \frac{1}{8k^2}n^2. \nonumber
	\eeqs
This is a contradiction since $G$ is missing $O(f(n))$ edges between vertex classes. Therefore $u$ or $v$ must belong to $S$.

We have shown that each edge in $E_I$ is incident with $S$, every vertex of $S$ is incident with at most $r(v) = O(f(n)^{1/2})$ edges from $E_I$, and $|S| = O(f(n)n^{-1})$. It follows that $|E_I| = O(f(n)^{3/2}n^{-1})$.
\end{proof}

The proof of Theorem \ref{stable} came in two parts. The first part bounded the number of neighbours a vertex can have inside its respective vertex class by considering whether there is a copy of $T_k(kt)$ in its neighbourhood. When we move to the regime of graphs contained in some $M_k(a,a,1)$, we can improve the argument by considering whether there is a copy of $T_k(kt)$ present in many of the neighbourhoods of the vertex's neighbours. Again, our arguments are guided by considering the examples giving lower bounds.

\begin{proof}[\bf Proof of Theorem \ref{stable2}]
Let $f$ be a function on the natural numbers such that $f(n) = o(n^2)$, let $H$ be a graph with a critical edge and $\chi(H) = k+1$ and suppose that $h$ is such that $H \subset M_k(h,h,1)$. Let $G$ be an $H$-free graph on $n$ vertices with $e(G) \ge t_k(n) - f(n)$. Let $\delta, \eta \in (0,\frac{1}{20000hk^2})$. Take a partition $(V_1,\ldots,V_k)$ of $V(G)$ which minimises the total number of edges inside vertex classes and let $E_I$ be the set of edges inside vertex classes. Furthermore, let $S = \{u \in V(G) : d(u) \le (1- \delta)n\frac{k-1}{k}\}$.

Carrying on from the end of the proof of Theorem \ref{stable}, we know that each edge $e \in E_I$ is incident with a vertex in $S$ and that $|S| = O(f(n)n^{-1})$. It therefore suffices to show that the maximum number of neighbours a vertex can have inside its own vertex class is $O(f(n)n^{-1})$.

Suppose without loss of generality that $v \in V_1$ has the maximum number of neighbours inside its own vertex class. Then for each $i$, let $A_i = V_i \cap \Gamma(v)$ and split each $A_i$ into $B_i = S \cap A_i$ and $C_i = A_i \setminus B_i$. Let us consider the size of the $C_i$. Suppose that $|C_i| \ge h$ for each $i$ and pick $h$-subsets $D_i \subset C_i$ for each $i$. Now for each $i \in [k]$, let
	\beqs
		W_i = \left\{x \in V_i \setminus (D_i \cup \{v\}) : \bigcup_{j \neq i}D_j \subset \Gamma(x)\right\}. \label{IIhelp}
	\eeqs
Note that for large enough $n$, each $u \in D_i$ is adjacent to all but at most $2\delta \frac{n}{k}$ vertices in $V_j$. Thus for large enough $n$, $|W_i| \ge \frac{n}{2k} - 2kh\delta \frac{n}{k} \ge \eta n$. So pick $\eta n$ vertices from each $W_i$ to form a set $Q$ of vertices and let $J$ be the subgraph of $G[Q]$ with all edges inside vertex classes deleted. Note that if $J$ contains a copy, $J[F]$, of $T_k(kh)$, then $G[\{v\} \cup \bigcup_{i \in [k]}D_i \cup F]$ will contain a copy of $M_k(h,h,1)$ and so will contain a copy of $H$, a contradiction. So we may apply Proposition \ref{crux} to give that for $n$ sufficiently large,
	\beqs
		e(J) \le t_k(k\eta n) - \frac{\eta^2}{2}n^2. \nonumber
	\eeqs
On the other hand, we know that there are $O(f(n))$ edges between vertex classes not present in $G$. Therefore, $e(J) \ge t_k(k\eta n) - O(f(n))$. We then have a contradiction since $f(n) = o(n^2)$. So there is a $j \in [k]$ such that $|C_j| < h$.

Now note since $B_j \subset L$, that $|B_j| = O(f(n)n^{-1})$ and so $|A_j| =O(f(n)n^{-1})$. Note that since we have taken the partition which minimises the total number of edges inside vertex classes, $|A_1| \le |A_j|$. We conclude that the maximum number of neighbours a vertex can have inside its own vertex class is $|A_1| = O(f(n)n^{-1})$.
\end{proof}

Theorem \ref{escrit} now follows as a direct corollary of Theorem \ref{stable}.

\begin{proof}[\bf Proof of Theorem \ref{escrit}]
Let $f(n) = o(n^2)$ be a function, and let $H$ be a graph with a critical edge and $\chi(H)$ = $k+1$. Let $G$ be an $H$-free graph on $n$ vertices with $e(G) \ge t_k(n) - f(n)$. By Theorem \ref{stable}, we can delete $O(f(n)^{3/2}n^{-1})$ edges from $G$ to form a $k$-partite graph $G'$ which has $t_k(n) - O(f(n))$ edges. Suppose that $V_1,\ldots,V_k$ is a vertex colouring of $G'$. If the size of two colour classes differ by $2t$, then the maximum number of edges possible in the graph $G'$ would be $t_k(n) - \Theta(t^2)$, and so two classes can differ in size by at most $O(f(n)^{1/2})$. It follows that $||V_i| - \frac{n}{k}| = O(f(n)^{1/2})$ for each $i$ and so a new graph $G''$ with equal class sizes can be formed by deleting the edges incident to $O(f(n)^{1/2})$ vertices. $G''$ has $t_k(n) - O(f(n)^{1/2}n)$ edges and has class sizes equal to that of the Tur\'{a}n graph. We can then attain the Tur\'{a}n graph by filling in the missing edges. To summarise, we deleted $O(f(n)^{1/2}n)$ edges to form $G''$ and then added $O(f(n)^{1/2}n)$ edges to reach the Tur\'{a}n graph.

\end{proof}

To end this section, we prove Theorem \ref{stableII} using a method similar to that used by K\H{o}v\'ari, S\'{o}s and Tur\'{a}n \cite{zara}.

\begin{proof}[\bf Proof of Theorem \ref{stableII}]
Let $f(n) = o(n^2)$ be a function, and let $H$ be a graph with a critical edge and $\chi(H)$ = $k+1$. Further suppose that $h, t$ and $a$ are natural numbers such that $H$ is contained in $M_k(t,h,a)$ and that $H$ is contained in no $M_k(b,b,1)$. Suppose that $G$ is an $H$-free graph on $n$ vertices such that $e(G) \ge t_k(n) - f(n)$. Take the $k$-partition $W_1,\ldots, W_k$ of $V(G)$ which minimises the total number of edges inside vertex classes. By the proof of Theorem \ref{stable} we know that there is a set $S$ of order $|S| = O(f(n)n^{-1})$ such that $G-S$ is a $k$-partite graph, that the minimum degree of the vertices in $V(G) \setminus S$ is $(1+o(1))n\frac{k-1}{k}$ and that the maximum number of edges inside a vertex class incident to a vertex in $S$ is $O(f(n)^{1/2})$. We further know that there are $O(f(n))$ edges missing between any pair $V_i, V_j$, where $V_l = W_l \setminus S$ for each $l\in [k]$. 

Let $E_I$ be the set of edges inside the vertex classes (so that $G$ is $|E_I|$ edges away from being $k$-partite). We assume that $|E_I| = \Omega(f(n)^{1-\frac{1}{bk}}n^{\frac{1}{bk}})$, else we are done. Note that if we delete a different set of edges $F \subset E$ to obtain a $k$-partite subgraph of $G$, then since we have taken the $k$-partition which minimises the total number of edges inside vertex classes, it must be the case that $|F| \ge |E_I|$. So to get an upper bound on $|E_I|$ consider deleting all the edges between vertices in $S$, of which there are $O(f(n)^2n^{-2})$, and then deleting the edges between each vertex $s \in S$ and one of the $V_i$, where $i$ may depend upon $s$. The best we could do (in terms of minimising edges deleted) by using this method is if for each $s \in S$ we deleted the edges between $s$ and a $V_i$ such that $|\Gamma(s) \cap V_i|$ is minimised. So if we let $e_I(s) = \min\{|\Gamma(s)\cap V_i| : i \in [k]\},$ we have an upper bound for the number of irregular edges in $G$.
	\beqs
		|E_I| \le \sum_{s \in S} e_I(s) + O(f(n)^2n^{-2}). \nonumber
	\eeqs
Recall that $|E_I| = \Omega(f(n)^{1-\frac{1}{bk}}n^{\frac{1}{bk}})$ and note that $f(n)^2n^{-2} = o(f(n)^{1-\frac{1}{bk}}n^{\frac{1}{bk}})$. This means that $\sum_{s \in S} e_I(s) = \Theta(E_I)$. Let $S' = \{s\in S: e_I(s) \ge 2h\}$. Since $|S| = O(f(n)n^{-1})$ the contribution to the sum of those vertices in $S \setminus S'$ is negligible and so
	\beqs
		\sum_{s \in S'} e_I(s) = \Theta(E_I). \label{IIcount3}
	\eeqs

Following the argument of Theorem \ref{stable2} from \eqref{IIhelp}, we see that if we pick $h$-subsets $D_i$ of each $V_i$, then there exists $C_i \in (V_i \setminus D_i)^{(t)}$ for each $i$ such that $G[C_i,D_j]$ and $G[C_i,C_j]$ are homomorphic to $K_{t,h}$ and $K_{t,t}$ respectively for $i \neq j$. Therefore if there is an $a$-set $A$ in $S$ such that the vertices share $h$ common neighbours in each $V_i$, then we can find a copy of $M_k(t,h,a)$ in $G$ which contradicts our initial assumption that $G$ is $H$-free. We will therefore count how many times an element of $V_1^{(h)} \times \ldots \times V_k^{(h)}$ is contained within the neighbourhood of a vertex in $S$.

Since no element of $V_1^{(h)} \times \ldots \times V_k^{(h)}$ can we contained within the neighbourhood of $a$ distinct vertices in $S$, we have an upper bound given by $a|V_1^{(h)} \times \ldots \times V_k^{(h)}| = O(n^{hk})$. On the other hand if we count over the vertices of $S$, we see that the neighbourhood of a vertex $s \in S$, contains
	\beqs
		\prod_{i \in [k]} \binom{|\Gamma(s) \cap V_i| }{h} \ge \binom{e_I(s) }{h}^k  \label{IIcount1}
	\eeqs
elements of $V_1^{(h)} \times \ldots \times V_k^{(h)}$. Summing over the vertices in and comparing to the upper bound given earlier, we see that
	\beqs
		\sum_{s \in S}\binom{e_I(s) }{ h}^k = O(n^{hk}). \label{IIcount2}
	\eeqs

Note that we may easily bound the left hand side of \eqref{IIcount2} by summing only over $S'$ and bounding $\binom{e_I(s)}{ h}^k$ below by $(\frac{e_I(s)}{h})^{hk}$ to get that
	\beqs
		\sum_{s \in S'} e_I(s)^{hk} = O(n^{hk}). \label{IIcount4}
	\eeqs

We can then bound the left hand side by applying H\"{o}lder's inequality to get that
	\beqs
		\sum_{s \in S'} e_I(s)^{hk} &\ge& |S'|^{1-hk} (\sum_{s \in S'} e_I(s))^hk \nonumber \\
		&\ge& (f(n)n^{-1})^{1-hk}(\sum_{s \in S'} e_I(s))^hk \nonumber \\
		&=& \Theta((f(n)n^{-1})^{1-hk}E_I^{hk}). \label{IIcount5}
	\eeqs
Combining \eqref{IIcount4} and \eqref{IIcount5}, we see that
	\beqs
		(f(n)n^{-1})^{1-hk}E_I^{hk} = O(n^{hk}), \nonumber
	\eeqs
and so after rearranging, the result follows.
\end{proof}

\section{Lower Bounds}\label{countergraphs}
In this section, we prove Propositions \ref{counter1} and \ref{propcount1}.
\begin{proof}[Proof of Proposition \ref{counter1}]
Fix $k \ge 2$ and let $f(n) = o(n^2)$ be a function with $f(n) \ge 2n$. Let $r = \frac{f(n)^{1/2}}{k^2}$ and $s = \frac{f(n)}{2n}$. For $n$ a large positive integer, consider the graph
	\beqs
		G := M_k(\frac{n-s-kr}{k},r,s). \nonumber
	\eeqs
Note that the numbers given for $G$ may not be integer valued. This can easily be fixed but we have left it as it is for clarity and ease of reading (this will also be true of the remainder of this paper). Furthermore label subsets of the vertices of $G$ as in Figure \ref{graphcounter1}.

\begin{minipage}{\textwidth}
\centering
\begin{tikzpicture}[thick,scale=0.6, every node/.style={scale=0.6}]
	\node at (0,0) (z1) {} ;
	\node[vertex2,position=90:{0.7} from z1] (x1) {$U$} ;
	\node[vertex2,position=180:{3} from z1] (y3) {$W_3$} ;
	\node[vertex2,position=90:{1.5} from y3] (y2) {$W_2$} ;
	\node[vertex2,position=90:{1.5} from y2] (y1) {$W_1$} ;
	\node[vertex2,position=90:{-4} from y3] (yk) {$W_k$} ;
	\node[vertex2,position=180:{4.5} from y3] (z3) {$V_3$} ;
	\node[vertex2,position=180:{4.5} from y2] (z2) {$V_2$} ;
	\node[vertex2,position=180:{3} from y1] (z1) {$V_1$} ;
	\node[vertex2,position=180:{3} from yk] (zk) {$V_k$} ;
	\draw[dashed] (y3) -- (yk) ;
	\draw (x1) -- (y1) ;
	\draw (x1) -- (y2) ;
	\draw (x1) -- (y3) ;
	\draw (x1) -- (yk) ;
	\draw (y1) -- (z2) ;
	\draw (y1) -- (z3) ;
	\draw (y1) -- (zk) ;
	\draw (y2) -- (z1) ;
	\draw (y2) -- (z3) ;
	\draw (y2) -- (zk) ;
	\draw (y3) -- (z2) ;
	\draw (y3) -- (z1) ;
	\draw (y3) -- (zk) ;
	\draw (yk) -- (z2) ;
	\draw (yk) -- (z3) ;
	\draw (yk) -- (z1) ;
	\draw (z1) -- (z2) ;
	\draw (z3) -- (z2) ;
	\draw (z1) -- (z3) ;
	\draw (z1) -- (zk) ;
	\draw (zk) -- (z2) ;
	\draw (zk) -- (z3) ;
\end{tikzpicture}
\captionsetup{font=footnotesize}
\captionof{figure}{$G$}\label{graphcounter1}
\end{minipage}

Now suppose that $H$ is graph with a critical edge and $\chi(H) = k+1$ and $H$ is not a subgraph of $M_k(a,a,a)$ for any $a$. Thus $H$ cannot be a subgraph of $G$.

We can obtain $T_k(n)$ from $G$ by adding the edges between the $W_i$ (${k \choose 2} r^2 \le \frac{f(n)}{2}$ edges) and changing the edges incident with $U$ ($ns = \frac{f(n)}{2}$ edges). In this process we add at most $f(n)$ edges and so $e(G) \ge t_k(n) - f(n)$. It is therefore enough to show that $G$ is $\Omega(f(n)^{3/2}n^{-1})$ edges away from being $k$-partite.

So let $Q$ be a $k$-partite subgraph of $G$ formed by deleting edges from $G$. If we have deleted a fraction $\frac{1}{8k^2}$ of the edges between some pair $(U,W_i), (W_i,V_j)$ or $(V_i,V_j)$ in forming $Q$ from $G$, then in all cases we have deleted at least $\frac{f(n)^{3/2}n^{-1}}{16k^4}$ edges. Otherwise if we pick a vertex uniformly at random from each $U, W_i$ and $V_j$ to form a copy of $M_k(1,1,1)$ within $G$, then we expect to have deleted less than half an edge on average from this subgraph when forming $Q$ (note that $e(M_k(1,1,1)) \le 4k^2$). It must then be the case that we can pick a vertex from each $U, W_i$ and $V_j$ to form a copy of $M_k(1,1,1)$ in $Q$. This contradicts $Q$ being $k$-partite and so there must be some pair $(U,W_i), (W_i,V_j)$ or $(V_i,V_j)$ between which we have deleted a fraction $\frac{1}{8k^2}$ of the edges. In all cases we must have deleted $\Omega(f(n)^{3/2}n^{-1})$ edges from $G$ and so $G$ must be $\Omega(f(n)^{3/2}n^{-1})$ edges away from being $k$-partite.
\end{proof}

For critical graphs contained within some $M_k(a,a,a)$ we will consider graphs with chromatic number $k+1$ such that every small subgraph has chromatic number at most $k$. We can construct an example of such a graph by adding more levels to the Mycielskian graph of a clique and blowing it up.

\begin{defn}
Let $a,b,c,l$ and $k$ be positive integers. Let $V_1,\ldots,V_k$ be sets of size $a$, let $W_1^1,\ldots W_k^1,\ldots,W_1^l,\ldots,W_k^l$ be sets of size $b$ and let $U$ be a set of size $c$ (and let all these sets be disjoint). Then the \it $l$-layer Mycielskian graph \it $M_k^{(l)}(a,b,c)$ has vertex set $V(G) \cup \bigcup_{i=1}^k V_i \cup \bigcup_{i\in [k],m\in [l]}W_i^m \cup U$ and edge set
	\beqs
		\bigcup_{i \neq j} K[V_i,V_j] \cup \bigcup_{i\neq j} K[V_j,W_i^1] \cup \bigcup_{i\neq j,m \in [l-1]} K[W_i^m,W_j^{m+1}] \cup \bigcup_{i \in k} K[W_i^l,U]. \nonumber
	\eeqs
\end{defn}

\begin{minipage}{\textwidth}
\centering
\begin{tikzpicture}[thick,scale=0.6, every node/.style={scale=0.6}]
	\node[vertex] at (0,0) (x1) {} ;
	\node[vertex, position=90:{1.7} from x1] (x2) {} ;
	\node[vertex, position=90:{-2.2} from x1] (x3) {} ;
	\node[position=180:{3} from x1] (dl) {} ;
	\node[vertex,position=90:{0.2} from dl] (y3) {} ;
	\node[vertex,position=90:{2.2} from dl] (y2) {} ;
	\node[vertex,position=90:{3.2} from dl] (y1) {} ;
	\node[vertex,position=90:{-0.8} from dl] (y4) {} ;
	\node[vertex,position=90:{-2.8} from dl] (y5) {} ;
	\node[vertex,position=90:{-3.8} from dl] (y6) {} ;
	\node[position=180:{3} from dl] (dl2) {} ;
	\node[vertex,position=90:{0.2} from dl2] (a3) {} ;
	\node[vertex,position=90:{2.2} from dl2] (a2) {} ;
	\node[vertex,position=90:{3.2} from dl2] (a1) {} ;
	\node[vertex,position=90:{-0.8} from dl2] (a4) {} ;
	\node[vertex,position=90:{-2.8} from dl2] (a5) {} ;
	\node[vertex,position=90:{-3.8} from dl2] (a6) {} ;
	\node[position=180:{3} from dl2] (dl3) {} ;
	\node[vertex,position=90:{0.2} from dl3] (b3) {} ;
	\node[vertex,position=90:{2.2} from dl3] (b2) {} ;
	\node[vertex,position=90:{3.2} from dl3] (b1) {} ;
	\node[vertex,position=90:{-0.8} from dl3] (b4) {} ;
	\node[vertex,position=90:{-2.8} from dl3] (b5) {} ;
	\node[vertex,position=90:{-3.8} from dl3] (b6) {} ;
	\node[position=180:{5} from dl3] (dr) {} ;
	\node[vertex,position=180:{1} from dr] (z2) {} ;
	\node[vertex,position=90:{1.7} from dr] (z1) {} ;
	\node[vertex,position=90:{-2.2} from dr] (z3) {} ;
	\draw (y1) -- (a3) ;
	\draw (y1) -- (a4) ;
	\draw (y1) -- (a5) ;
	\draw (y1) -- (a6) ;
	\draw (y2) -- (a3) ;
	\draw (y2) -- (a4) ;
	\draw (y2) -- (a5) ;
	\draw (y2) -- (a6) ;
	\draw (y3) -- (a1) ;
	\draw (y3) -- (a2) ;
	\draw (y3) -- (a5) ;
	\draw (y3) -- (a6) ;
	\draw (y4) -- (a1) ;
	\draw (y4) -- (a2) ;
	\draw (y4) -- (a5) ;
	\draw (y4) -- (a6) ;
	\draw (y5) -- (a3) ;
	\draw (y5) -- (a4) ;
	\draw (y5) -- (a1) ;
	\draw (y5) -- (a2) ;
	\draw (y6) -- (a3) ;
	\draw (y6) -- (a4) ;
	\draw (y6) -- (a1) ;
	\draw (y6) -- (a2) ;
	\draw (b1) -- (a3) ;
	\draw (b1) -- (a4) ;
	\draw (b1) -- (a5) ;
	\draw (b1) -- (a6) ;
	\draw (b2) -- (a3) ;
	\draw (b2) -- (a4) ;
	\draw (b2) -- (a5) ;
	\draw (b2) -- (a6) ;
	\draw (b3) -- (a1) ;
	\draw (b3) -- (a2) ;
	\draw (b3) -- (a5) ;
	\draw (b3) -- (a6) ;
	\draw (b4) -- (a1) ;
	\draw (b4) -- (a2) ;
	\draw (b4) -- (a5) ;
	\draw (b4) -- (a6) ;
	\draw (b5) -- (a3) ;
	\draw (b5) -- (a4) ;
	\draw (b5) -- (a1) ;
	\draw (b5) -- (a2) ;
	\draw (b6) -- (a3) ;
	\draw (b6) -- (a4) ;
	\draw (b6) -- (a1) ;
	\draw (b6) -- (a2) ;
	\draw (b1) -- (z2) ;
	\draw (b1) -- (z3) ;
	\draw (b2) -- (z2) ;
	\draw (b2) -- (z3) ;
	\draw (b3) -- (z1) ;
	\draw (b3) -- (z3) ;
	\draw (b4) -- (z1) ;
	\draw (b4) -- (z3) ;
	\draw (b5) -- (z2) ;
	\draw (b5) -- (z1) ;
	\draw (b6) -- (z2) ;
	\draw (b6) -- (z1) ;
	\draw (z1) -- (z2) ;
	\draw (z2) -- (z3) ;
	\draw (z3) -- (z1) ;
	\draw (x1) -- (y1) ;
	\draw (x1) -- (y2) ;
	\draw (x1) -- (y3) ;
	\draw (x1) -- (y4) ;
	\draw (x1) -- (y5) ;
	\draw (x1) -- (y6) ;
	\draw (x2) -- (y1) ;
	\draw (x2) -- (y2) ;
	\draw (x2) -- (y3) ;
	\draw (x2) -- (y4) ;
	\draw (x2) -- (y5) ;
	\draw (x2) -- (y6) ;
	\draw (x3) -- (y1) ;
	\draw (x3) -- (y2) ;
	\draw (x3) -- (y3) ;
	\draw (x3) -- (y4) ;
	\draw (x3) -- (y5) ;
	\draw (x3) -- (y6) ;
\end{tikzpicture}
\captionsetup{font=footnotesize}

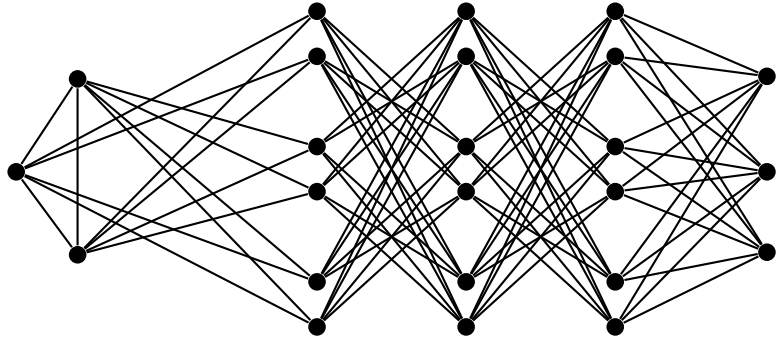
\captionof{figure}{$M_3^{(3)}(1,2,3)$}\label{neartriangle2}
\end{minipage}

\begin{proof}[Proof of Proposition \ref{propcount1}]
Fix $k \ge 2$ and let $f(n) = o(n^2)$ be a function. Let $s=\frac{f(n)}{2n}$. Suppose that $H$ is a graph with a critical edge on $N$ vertices with $\chi(H) = k+1$. For $n$ a large positive integer, consider the graph $G$,
	\beqs
		G = M_k^{(N)}(\frac{n-(Nk+1)s}{k},s,s). \nonumber
	\eeqs	
Furthermore label subsets of the vertices of $G$ as in Figure \ref{graphcounter1} (we have drawn an example with $k=4$).

\begin{minipage}{\textwidth}
\centering
\begin{tikzpicture}[thick,scale=0.6, every node/.style={scale=0.6}]
	\node at (0,0) (z1) {} ;
	\node[vertex2,position=90:{0.7} from z1] (x1) {$U$} ;
	\node[vertex2,position=180:{3} from z1] (y3) {$W^1_3$} ;
	\node[vertex2,position=90:{1.5} from y3] (y2) {$W^1_2$} ;
	\node[vertex2,position=90:{1.5} from y2] (y1) {$W^1_1$} ;
	\node[vertex2,position=90:{-4} from y3] (y4) {$W^1_4$} ;
	\node[vertex2,position=180:{3} from y3] (y23) {$W^2_3$} ;
	\node[vertex2,position=90:{1.5} from y23] (y22) {$W^2_2$} ;
	\node[vertex2,position=90:{1.5} from y22] (y21) {$W^2_1$} ;
	\node[vertex2,position=90:{-4} from y23] (y24) {$W^2_4$} ;
	\node[vertex2,position=180:{3} from y23] (y33) {$W^N_3$} ;
	\node[position=90:{0.75} from y23] (dummy1) {} ;
	\node[position=90:{0.75} from y33] (dummy2) {} ;
	\node[vertex2,position=90:{1.5} from y33] (y32) {$W^N_2$} ;
	\node[vertex2,position=90:{1.5} from y32] (y31) {$W^N_1$} ;
	\node[vertex2,position=90:{-4} from y33] (y34) {$W^N_4$} ;
	\node[vertex2,position=180:{4.5} from y33] (z3) {$V_3$} ;
	\node[vertex2,position=180:{4.5} from y32] (z2) {$V_2$} ;
	\node[vertex2,position=180:{3} from y31] (z1) {$V_1$} ;
	\node[vertex2,position=180:{3} from y34] (z4) {$V_4$} ;
	\path (dummy1) -- node[auto=false]{\Huge \ldots} (dummy2);
	\draw (x1) -- (y1) ;
	\draw (x1) -- (y2) ;
	\draw (x1) -- (y3) ;
	\draw (x1) -- (y4) ;
	\draw (y21) -- (y2) ;
	\draw (y21) -- (y3) ;
	\draw (y21) -- (y4) ;
	\draw (y22) -- (y1) ;
	\draw (y22) -- (y3) ;
	\draw (y22) -- (y4) ;
	\draw (y23) -- (y2) ;
	\draw (y23) -- (y1) ;
	\draw (y23) -- (y4) ;
	\draw (y24) -- (y2) ;
	\draw (y24) -- (y3) ;
	\draw (y24) -- (y1) ;
	\draw (y31) -- (z2) ;
	\draw (y31) -- (z3) ;
	\draw (y31) -- (z4) ;
	\draw (y32) -- (z1) ;
	\draw (y32) -- (z3) ;
	\draw (y32) -- (z4) ;
	\draw (y33) -- (z2) ;
	\draw (y33) -- (z1) ;
	\draw (y33) -- (z4) ;
	\draw (y34) -- (z2) ;
	\draw (y34) -- (z3) ;
	\draw (y34) -- (z1) ;
	\draw (z1) -- (z2) ;
	\draw (z3) -- (z2) ;
	\draw (z1) -- (z3) ;
	\draw (z1) -- (z4) ;
	\draw (z4) -- (z2) ;
	\draw (z4) -- (z3) ;
\end{tikzpicture}
\captionsetup{font=footnotesize}
\captionof{figure}{$G$}\label{graphcounter2}
\end{minipage}

Note that if we delete $U$ or $\bigcup_i V_i$, then we are left with a $k$-partite and a bipartite graph respectively. It must then be the case that any subgraph $J \subset G$ with chromatic number $k+1$ must contain vertices in both $U$ and $\bigcup_i V_i$ and so must contain at least $N + 2$ vertices. It follows that $G$ is an $H$-free graph since $|H| = N$.

We can obtain $T_k(n)$ from $G$ by adding the edges between the $W_i^j$ ($N*{k \choose 2}s^2 \le Nk^2f(n)^2n^{-2} \le \frac{f(n)}{2}$ edges) and changing the edges incident with $U$ ($ns = \frac{f(n)}{2}$ edges). In this process we add at most $f(n)$ edges and so $e(G) \ge t_k(n) - f(n)$. Therefore if we can show that $G$ is $\Omega(f(n)^2n^{-2})$ edges away from being $k$-partite, then we will be done.

So let $Q$ be a $k$-partite subgraph of $G$ formed by deleting edges from $G$. If we have deleted a fraction $\frac{1}{8N^2k^2}$ of the edges between some pair $(U,W_i^1),(W_i^l,W_j^{l+1}), (W^N_i,V_j)$ or $(V_i,V_j)$ in forming $Q$ from $G$, then in all cases we have deleted at least $\frac{f(n)^{2}n^{-2}}{16N^2k^2}$ edges. Otherwise if we pick a vertex uniformly at random from each $U, W_i^l$ and $V_j$ to form a copy of $M_k^{(N)}(1,1,1)$ within $G$, then we expect to have deleted less than half an edge on average from this subgraph when forming $Q$ (Note that $e(M_k^{(N)}(1,1,1)) \le 4N^2k^2$). It must then be the case that we can pick a vertex from each $U, W^{(l)}_i$ and $V_j$ to form a copy of $M_k^{(N)}(1,1,1)$ in $Q$. This contradicts $Q$ being $k$-partite and so there must be some pair $(U,W^l_i),(W^l_i,W_i^{l+1})$ or $(W_i,V_j)$ between which we have deleted a fraction $\frac{1}{8N^2k^2}$ of the edges. In all cases we must have deleted $\Omega(f(n)^{3/2}n^{-1})$ edges from $G$ and so $G$ is $\Omega(f(n)^2n^{-2})$ edges away from being $k$-partite.
\end{proof}

\section{Conclusion}\label{gap}
We have given bounds that are tight to within a constant factor for graphs with a critical edge that are not contained in any $M_k(a,a,a)$ and also graphs that are contained in some $M_k(a,a,1)$. It would be interesting to have even sharper bounds. For instance, is it possible to get an exact result for Theorem \ref{escrit} or Theorem \ref{stable}?

The other cases appear more difficult to handle. Theorem \ref{stableII} shows that for graphs contained in some $M_k(a,b,c)$ (but not with $c=1$), we can improve on the $O(f(n)^{3/2}n^{-1})$ upper bound of Theorem \ref{stable}. The arguments used in the proof of Theorem \ref{stableII} are rather crude, and it seems likely that stronger results should hold, at least when $f(n)$ is quite large. For instance, what can we say if $f(n) \ge n^{2- \epsilon}$ for small $\epsilon = \epsilon(H)$?

For some graphs we can improve on the $\Omega(f(n)^2n^{-2})$ lower bound of Proposition \ref{propcount1}. Let $r = f(n)^{1/2}$ and $s=f(n)n^{-1}$. Consider the graph
	\beqs
		G = M_k(\frac{n-s-kr}{k},r,s), \nonumber
	\eeqs
and label the subsets $U, W_1,\ldots,W_k,V_1,\ldots,V_k$ as in Figure \ref{graphcounter1}. Suppose we wanted to avoid a copy of $M_k(1,1,2)$. Then it would be sufficient to change the edges between $U$ and $W_1\cup \cdots \cup W_k$ so that for each pair of vertices $u_1 \neq u_2 \in U$, there is a $j \in [k]$ so that $u_1$ and $u_2$ have no common neighbour in $W_j$.

Let $q = \lceil (\frac{f(n)}{n} +1)^{\frac{1}{k}} \rceil$ and for each $i \in [k]$, let $W_i^{(0)},\ldots,W_i^{(q-1)}$ be disjoint subsets of $W_i$ all of size $\lfloor \frac{f(n)^{1/2}}{q} \rfloor$. Let the vertices of $U $ be $u_1,\ldots,u_r$, where $r = \lceil f(n)n^{-1} \rceil$. Each $j \in [r]$ can be expressed as a $q$-ary number with $k$ digits
	\beqs
		j = \sum_{i\in [k]} a_i q^{i-1}, \label{q-ary}
	\eeqs
where $a_i \in \{0,\ldots,q-1\}$ for each $i$. Now form a new graph $G'$ from $G$, where each vertex $u_j$ in $U$ has neighbourhood $\bigcup_{i \in [k]}W_i^{(a_i)}$ where the $a_i$ are as in \eqref{q-ary}. The graph $G'$ does not contain $M_k(c,b,a)$ for any $a,b$ and $c$ with $a \ge 2$, but is $\Theta(f(n)^{3/2 - \frac{1}{k}}n^{-1 + \frac{1}{k}})$ edges away from being $k$-partite. This improves a little on the bound given by Proposition \ref{propcount1}.

Proposition \ref{counter1} tells us that there are constants $C_1,C_2$ and a $K_{k+1}$-free graph $G$ with $e(G) \ge t_k(n) - C_1n$ which is at least $C_2n^{1/2}$ edges away from being $k$-partite. On the other hand, if $e(G) \ge t_k(n) -\frac{n}{k} + O(1)$ then, as noted above, a result of Simonovits \cite[p. 282]{Sim4} shows that $G$ must be $k$-partite. It would be interesting to know what happen in the range in between.

In this paper we have discussed graphs $H$ with a critical edge. It would be interesting to get sharp results for all graphs both for the Erd\H{o}s-Simonovits problem of distance from the Tur\'{a}n graph, and for the problem of the distance from being $k$-partite.

Finally, we note that the condition $f(n) = o(n^2)$ in Theorems \ref{stable}, \ref{stable2}, \ref{stableII} and Propositions \ref{counter1}, \ref{propcount1} can be replaced by the condition that $f(n) \le \epsilon n^2$ for some sufficiently small $\epsilon$ with only minor changes to the proof.

\bigskip \noindent {\bf Acknowledgements:} We would like to thank the referees for their helpful comments.


\begin{thebibliography}{[AHU]}

\bibitem{erd-sim}
P. Erd\H{o}s, Some recent results on extremal problems in graph theory, \em Proc. Theory of Graphs Internl. Symp. Rome \em (1966), 118--123.

\bibitem{radem}
P. Erd\H{o}s, On a theorem of Rademacher-Tur\'{a}n, \em Illinois Journal of Mathematics \em \textbf{6} (1962), 122--127.

\bibitem{Sim2}
P. Erd\H{o}s and M. Simonovits, A limit theorem in graph theory, \em Studia Sci. Math. Hungar. \em \textbf{1} (1966), 51--57.

\bibitem{Sim3}
P. Erd\H{o}s and M. Simonovits, On a valence problem in extremal graph theory, \em Discrete Mathematics \em \textbf{5} (1973), 323--334.

\bibitem{erd-stone}
P. Erd\H{o}s and A. H. Stone, On the structure of linear graphs, \em Bull. Amer. Math. Soc. \em \textbf{52} (1946), 1087--1091. 

\bibitem{Furedi}
Z. F{\"u}redi, A proof of the stability of extremal graphs, Simonovits' stability from Szemer{\'e}di's regularity, \em Journal of Combinatorial Theory, Series B \em \textbf{115} (2015), 66--71.

\bibitem{gyori}
E. Gy\H{o}ri, On the number of edge disjoint cliques in graphs of given size, \em Combinatorica \em \textbf{11} (1991), 231--243.

\bibitem{hajos}
G. Haj\'{o}s, Uber eine Konstruktion nicht $n$-f\"{a}rbbarer Graphen, \em Wiss. Z. Martin-Luther-Univ. Halle-Wittenberg Math.-Natur. Reihe \em \textbf{10} (1961), 116--117.

\bibitem{toft}
D. Hanson and B. Toft, $k$-saturated graphs of chromatic number at least $k$, \em Ars Combin. \em \textbf{31} (1991), 159--164.

\bibitem{zara}
T. K\H{o}v\'ari, V. S\'{o}s and P. Tur\'{a}n, On a problem of K. Zarankiewicz, \em Colloquium Mathematicae \em \textbf{3} (1954), 50--57.

\bibitem{cycles}
N. Morrison, A. Roberts and A. Scott. Maximising the number of cycles in graphs with forbidden subgraphs, \em In preparation. \em

\bibitem{Myc}
J. Mycielski. Sur le coloriage des graphes, \em Colloq. Math. \em \textbf{3} (1955), 161--162.

\bibitem{Liana}
S. Norin and L. Yepremyan. Tur\'{a}n numbers of extensions, \em arXiv:1510.04689. \em

\bibitem{Oleg}
O. Pikhurko, J. Sliacan and K. Tyros. Strong forms of stability from flag algebra calculations, \em arXiv:1706.02612. \em

\bibitem{Sim1}
M. Simonovits. Extremal graph problems with symmetrical extremal graphs. Additional chromatic conditions, \em Discrete Mathematics \em \textbf{7} (1974), 349--376.

\bibitem{Sim4}
M. Simonovits. A method for solving extremal problems in graph theory, stability problems, in \em Theory of Graphs (Proc. Colloq. Tihany, 1966), \em Academic Press, New York (1968), 279--319.

\bibitem{Liana2}
Liana Yepremyan, personal communication.

\end{thebibliography}
\end{document}